\documentclass[leqno, 12pt]{article}
\usepackage{amsmath,amsthm,amssymb,amscd,ascmac,amsfonts,multicol,tabularx}

\newtheorem{main thm}{\quad Theorem}
\newtheorem{df}{\quad Definition}[section]
\newtheorem{thm}[df]{\quad Theorem}
\newtheorem{lem}[df]{\quad Lemma}

\newtheorem{prop}[df]{\quad Proposition}
\newtheorem{rem}[df]{\quad Remark}

\newcommand{\N}{\mathbb{N}}
\newcommand{\Z}{\mathbb{Z}}
\newcommand{\Q}{\mathbb{Q}}

\newcommand{\F}{\mathbb{F}}

\newcommand{\x}{\times}

\newcommand{\up}{\stackrel}
\newcommand{\ds}{\displaystyle}

\DeclareMathOperator*{\zetaprod}{{\displaystyle{\prod\kern-1.45em\coprod}}}

\begin{document}
\begin{center}
{\Large {On conditions for $\rho$-value is 1 or not of complete family of pairing-friendly elliptic curves}
}\\
\medskip{Keiji OKANO}
\footnote[0]{\textit{Key Words}. Pairing-friendly elliptic curves, Pairing-based cryptography, Embedding degree.} \\
\end{center}


\begin{abstract}
We study whether a complete family of pairing friendly elliptic curves has a $\rho$-value $1$ or not.
We show that, in some cases, $\rho$-values are not to be $1$.
\end{abstract}


\section{Introduction}
The security of public-key cryptosystems is based on mathematical problems 
which seem not to have any efficiently computable solutions.
For example, RSA cryptosystems are based on the fact that 
we do not have any effective way of prime factor decomposition for large numbers.
Also, there are other cryptosystems which rely on the problem called the discrete logarithm problem (DLP).
If elements $a, b $ of a finite cyclic group with large order are given,
then the DLP is to give the solution $x$ of $a^x=b$ if it exists.
Computing the solution of this problem is apparently difficult.
Standard elliptic curve cryptosystems and cryptographic schemes which are based on pairings of elliptic curves depend their security on the DLP.
The standard elliptic curve cryptosystems are proposed by Koblitz \cite{Kob} and is also called elliptic ElGamal cryptosystems.
In the sense of security,
the cryptosystems using elliptic curves are much safer and have many tools for encryption
than using DLP on finite fields.
Moreover, they can be applied with small bit sizes.
Therefore, they are well studied.

In recent years, pairing-based cryptographic schemes have suggested.
They fit many new and novel protocols including ID-based encryption and one-round three-way key change.
The cryptographic schemes based on pairings of elliptic curves which we reference in \S \ref{ps-friendly} in this paper were suggested by Boneh-Franklin \cite{BF} and Sakai-Ohgishi-Kasahara \cite{SOK00}.
One of the features of them is that they require so-called pairing-friendly elliptic curves which have special properties,
whereas elliptic ElGamal cryptosystems can be implemented by using almost randomly generated elliptic curves.
Many strategies of constructing pairing-friendly elliptic curves have been proposed.
We define a parameter $\rho$ that represents how the given curves are close to the ideal in pairing-based cryptographic schemes.
This parameter expresses the ratio of $\log q$ and $\log r$ 
where $q$ is the order of the definition field of a given curve $E$ and 
$r$ is the prime order subgroup of the curve:
$\rho(E)=\log q/\log r$.
If $\rho(E)=1$, 
then the curve is ideal.
However, it is known that such case and the cases where $\rho$ is close to $1$ are very rare.

In this paper, after introducing the method constructing a family of pairing-friendly elliptic curves given by Brezing and Weng, 
we study a parameter $\rho(t(x), r(x), q(x))$ defined for such a family (Definition \ref{family of p-f cv}) which is different from the $\rho$-value above.
The case where the $\rho$-value equals $1$ is also ideal while most of the cases are not.
Only one example of $\rho(t(x), r(x), q(x))=1$ which is constructed by Barreto-Naehrig \cite{BN} is known (Remark \ref{BN}).
We propose a mathematical problem ``to give conditions that the $\rho$-value equals or is close to $1$''.
We denote the Euler function and the $k$th cyclotomic polynomial by $\varphi(x)$ and $\Phi_k(x)$, respectively.
Then our main theorems give many sufficient conditions of that $\rho$-values are not $1$ if we take $r(x)$ as a cyclotomic polynomial:

\begin{thm}\label{main thm 1}
Let $k$ be a positive integer and $D$ a square-free positive integer.
Define 
\[
r(x):=\Phi_k(x).
\]
Suppose that $(t(x), r(x), q(x))$ parameterizes a complete family of elliptic curves with CM-discriminant $D$ and embedding degree $k$ with respect to $r(x)$ 
{\rm (}these terms are defined in \S {\rm 2)}.
If one of the following holds, then $\rho(t(x), r(x), q(x)) \neq 1$.
\\
{\rm (i)} 
The degree $k \in \{1,2,p,2p\}$ for some prime number $p$.
{\rm (}If $p$ is odd, then $p$ satisfies that $p \equiv 3$ {\rm mod} $4$.{\rm )}
\\
{\rm (ii)} 
Conditions $k \in \{pQ,2pQ \}$, $D=p$, $t(x)=x+1$ and
\[
(p-2)Q+1<\varphi(k).
\]
Here, $p \ge 7$ be an odd prime number such that 
$p \equiv 3$ {\rm mod} $4$,
and
$Q \ge 2$ is an integer.
\end{thm}

For example, the inequality in (ii) holds if $k=pq$ and $p \le q$.

\begin{thm}\label{main thm 2}
Let $k$, $d$ be positive integers and $D$ a square-free positive integer.
Define
\[
r(x):=\Phi_{dk}(x),\ 
t(x):=x^{dg}+1,\ 
dg<\varphi(dk),
\]
where $\gcd(g,k)=1$.
Suppose that $(t(x), r(x), q(x))$ parameterizes a complete family of elliptic curves with embedding degree $k$ and CM-discriminant $D$ with respect to $r(x)$.
Then the following hold.
\\
{\rm (i)} 
Whether the $\rho$-value is $1$ or not can be reduced to the case where $\gcd(d,k)=1$ and $d$ is square-free 
by removing $\gcd(d, k)$ and the square-factor from $d$.
\\
{\rm (ii)} 
If one of the following holds, then $\rho(t(x), r(x), q(x)) \neq 1$.
\\
\quad {\rm (ii-a)} 
The degree $k=3, 6$.
\\
\quad {\rm (ii-b)} 
The degree $k$ has a square factor.
\end{thm}



\section{Family of Pairing-Friendly Elliptic Curves}\label{ps-friendly}
In this section, we briefly explain cryptosystems of pairing-based elliptic curves.
After that, we describe the strategy of constructing families of pairing-friendly elliptic curves as proposed by Brezing and Weng \cite{BW}.
We use the notation $\N$, $\Z$, $\Q$ as the set of positive integers, rational integers and rational numbers, respectively.
We describe $\F_q$ as a finite field with order $q$.
Pairing-based cryptographic systems require a non-degenerate alternative pairing which can be efficiently computed.
For example, the most common pairings used in applications are Weil and Tate pairings.
We define the embedding degree with respect to a cyclic group of order $r$ as the degree of the extension field which the pairing maps to.
In other words,

\begin{df}\label{embed deg}
Let be $q$ a power of a prime number, $E$ an elliptic curve over $\F_q$, and $r$ a prime number such that $r \mid \# E(\F_q)$, $r \nmid q$.
Here $E(\F_q)$ stands for the points of $E$ over $\F_q$.
The embedding degree of $E$ with respect to $r$ is the smallest positive integer $k$ such that
$r \mid \# \F_{q^k}^\x$.
\end{df}

Suitable curves for pairing-based cryptographic systems require a subgroup with large prime order $r$ and an efficiently computable pairing with small embedding degree $k$.
Therefore, for many embedding degrees and prime numbers, we should construct elliptic curves having such properties.
We remark that the embedding degree of supersingular elliptic curves is at most $6$.
Hence, in practice, ordinary elliptic curves are used.
The order of an elliptic curve $E$ over $\F_q$ is given by
\[
\# E(\F_q)=
q+1-t,
\]
where $t \in \Z$ is the trace of Frobenius map.

\begin{lem}[{\cite[Proposition 2.4]{FST}}]
Assume that $r \nmid kq$.
Then the condition of Definition {\rm \ref{embed deg}} is equivalent to
$
\Phi_k(t-1) \equiv 0
\ 
{\rm mod}\ r
$.
Here, $\Phi_k(x)$ is the $k$th cyclotomic polynomial.
\end{lem}

We describe the CM-method as proposed by Atkin and Moran, which is the strategy of constructing elliptic curves with given parameters.

\begin{thm}[Atkin-Moran \cite{AM}]
Let $k$ be a positive integer.
Suppose that there are some $t,r,q$ satisfying the following properties:
\\
{\rm (i)}
$q$ is a power of a prime number.
\\
{\rm (ii)}
$r$ is a prime such that $r \nmid k$.
\\
{\rm (iii)}
$r \mid q+1-t$, in other words, there is $h \in \N$ such that $rh=q+1-t$.
\\
{\rm (iv)}
$r \mid q^k-1$, and $r \nmid q^i-1$ $(1 \le i <k)$.
\\
{\rm (iv)}
There exist some $y \in \Z$ and some square-free positive integer $D$ such that an equation
$Dy^2=4q-t^2$, i.e., $Dy^2=4rh-(t-2)^2$
holds.
\\
{\rm (}$D$ is called a CM-discriminant{\rm )}
Then there exists an elliptic curve $E$ over $\F_q$ which satisfies the followings:
\\
{\rm (a)}
$\# E(\F_q) = q+1-t$ and there is a subgroup of $E(\F_q)$ with prime order $r$.
\\
{\rm (b)}
The embedding degree with respect to $r$ is $k$.
\end{thm}

In practice, the method can construct curves over finite fields when $D<10^{12}$ \cite{Sut}.
Suppose that an elliptic curve $E/\F_q$ has an embedding degree $k$ with respect to a prime number $r$.
For applying it to the pairing-based cryptography with sufficiently security level, Freeman-Scott-Teske \cite{FST} defined ``pairing-friendly" curves.
Namely, $E$ is a pairing-friendly if the following two conditions hold:
$
r \ge \sqrt{q}$,
$
k<(\log_2 r)/{8}.
$

One of known algorithms of constructing pairing-friendly elliptic curves is the Cocks-Pinch method.
The Brezing-Weng method which we refer later (Theorem \ref{BW}) basically uses the Cocks and Pinch idea over polynomials.
Moreover, define
\begin{eqnarray}\label{rho of p-f cv}
\rho(t,r,q):=\frac{\log q}{\log r}.
\end{eqnarray}
The value 
effects the superiority of our cryptography.
The case where $\rho(t,r,q)$ equals (or is close to) $1$ is ideal.
One of the aims in the study of pairing-based cryptography is to seek for curves with such values.
We refer that the curves produced by Cocks-Pinch method tend to have $\rho$-values around $2$.

\vspace{0.3cm}

\noindent
{\bfseries Brezing-Weng method}
\\
For applications, we would like to be able to construct curves of specified bit size.
To end this, we describe families of pairing-friendly curves for which the parameters $t$, $r$, $q$  are given as polynomials $t(x),r(x),q(x)$ in terms of a parameter $x$.
We give the idea of Brezing and Weng in \cite{BW}.
According to \cite{FST}, we define the following which is based on the conjecture of Bouniakowski and Schnzel \cite[p.323]{Lang}.

\begin{df}\label{rep primes}
{\rm (i)} 
Let $f(x)$ be a polynomial in $\Q[x]$.
If there is some $a \in \Z$ such that $f(a) \in \Z$, then we say $f(x)$ represents integers.
\\
{\rm (ii)} 
Assume that a non-constant irreducible polynomial $f(x) \in \Q[x]$ represents integers.
If $f(x)$ has a positive leading coefficient and 
\[
{\rm gcd}(\{ f(x) |\; x\ such\ that\ f(x) \in \Z\})=1,
\]
then we say $f(x)$ represents primes.
\end{df}

Bouniakowski, Schnzel and some mathematicians conjectured that if $f(x)$ represents primes, then $f(x)$ has infinitely many prime values.

\begin{df}\label{family of p-f cv}
Let $k \in \N$ and $D$ a positive square-free integer.
Suppose that triple of non-zero polynomials $(t(x),r(x),q(x)) \in \Q[x]^{3}$ satisfies the following conditions:
\\
{\rm (i)} 
$r(x)$ represents primes.
\\
{\rm (ii)} 
$q(x)$ is a power of a polynomial representing primes.
\\
{\rm (iii)} 
$r(x) \mid q(x)+1-t(x)$, i.e., there exists $h(x) \in \Q[x]$ such that $h(x)r(x)=q(x)+1-t(x)$.
\\
{\rm (iv)} 
$r(x) \mid \Phi_k(t(x)-1)$.
\\
{\rm (v)} 
There is some $y(x) \in \Q[x]$ such that
$Dy(x)^2
=
4q(x)-t(x)^2
=
4h(x)r(x)-(t(x)-2)^2$.
\\
Then we say that $(t(x),r(x),q(x))$ parameterizes a complete family of pairing-friendly elliptic curves with embedding degree $k$ and CM-discriminant $D$.
Moreover, we define
$
\rho(t(x),r(x),q(x)):=
{\ds \lim_{x \to \infty}}\frac{\log q}{\log r}
=
\frac{\deg q(x)}{\deg r(x)}.
$
\end{df}

The definition of $\rho(t(x),r(x),q(x))$ is different from (\ref{rho of p-f cv}).
In addition, we note that it may happen that $(t(x),r(x),q(x))$ satisfying Definition \ref{family of p-f cv} does not lead to any explicit examples of elliptic curves:
for example, $r(x)$, $q(x)$ may never have integer values simultaneously.
But, in known examples used in applications, it does not happen.
We note that $t(x)$, $y(x)$ are determined up to modulus $r(x)$.
We have
\begin{eqnarray}\label{rho,y,t}
\rho(t(x),r(x),q(x))
=
\frac{2 \max\{ \deg y, \deg t\}}{\deg r}.
\end{eqnarray}

Next, we describe the Brezing-Weng method.
Denote the set of the $k$th primitive roots of unity by $\mu_k \subset \overline{\Q}$.
Here $\overline{\Q}$ is an algebraic closure of $\Q$.

\begin{thm}[Brezing-Weng \cite{BW}]\label{BW}
Let $k \in \N$ and $D$ a positive square-free integer.
Then execute the following steps.
\\
1. 
Choose an algebraic number field $K$ which contains $\Q(\mu_k)$ and $\Q(\sqrt{-D})$.
\\
2. 
Find an irreducible polynomial $r(x) \in \Z[x]$ with positive leading coefficient and an isomorphism such that $\Q[x]/(r(x)) \up{\sim}{\to} K$.
\\
3. 
Let $t(x)-1 \in \Q[x]$ be a polynomial mapping to a fixed element $\zeta_k \in \mu_k$ by the above isomorphism.
\\
4. 
Let $y(x) \in \Q[x]$ be a polynomial mapping to $\frac{\zeta_k-1}{\sqrt{-D}} \in K$ by the above isomorphism.
\\
5. 
Let $q(x) \in \Q[x]$ be given by
$q(x):=(Dy(x)^2+t(x)^2)/{4}$.
\\
If $q(x)$ represents primes and $y(x)$ represents integers,
then the triple $(t(x),r(x),q(x))$ parameterizes a complete family of elliptic curves with embedding degree $k$ and CM-discriminant $D$.
\end{thm}

\begin{rem}\label{BN}
{\rm The choice of $r(x)$ is an important part in this algorithm.
In \cite{BW}, Brezing and Weng calculated the cases where $D=1$ or $3$ and $K$'s are cyclotomic fields.
On the other hand, \cite{FST} collected many examples.
Barreto and Naehrig \cite{BN} gave an example of $\rho(t(x),r(x),q(x))=1$ with $k=12$, $D=3$:
\[
t(x)=6x^2+1,\ 
r(x)=36x^4+36x^3+18x^2+6x+1,\ 
q(x)=36x^4+36x^3+24x^2+6x+1.
\]
This is the {\it only one} known example of $(t(x),r(x),q(x))$ which parameterizes a complete family of curves with $\rho$-value $1$.
}
\end{rem}

\begin{rem}
{\rm We can also consider another problem which is to find a family of elliptic curves, so-called sparse family.
In detail, to find a family which has infinite many integral solutions $(x,y)$ of
\begin{eqnarray}\label{sparse}
Dy^2=4h(x)r(x)-(t(x)-2)^2,
\end{eqnarray}
instead of finding the equation
$Dy(x)^2=4h(x)r(x)-(t(x)-2)^2$. 
Many results about this problem are also known.
We refer the following examples for reference.
}


\begin{prop}[Miyaji-Nakabayashi-Takano \cite{MNT}, Freeman \cite{Free}]
The following pairs of polynomials satisfy {\rm (i), \ldots, (iv)} in Definition {\rm \ref{family of p-f cv}} and {\rm (\ref{sparse})}.
Moreover, their $\rho$-values
$
\rho(t(x),r(x),q(x))=
{
\frac{\deg q(x)}{\deg r(x)}
}
$
are $1$.
\\
{\rm (i)}
$k=3${\rm :}
$
t(x)=-1\pm6x,\ r(x)=12x^2 \pm 6x+1,\ q(x)=12x^2-1
.$
\\
{\rm (ii)}
$k=4${\rm :}
$t(x)=-x\ ({\rm resp.\ }x+1),\ r(x)=x^2+2x+2\ ({\rm resp.\ }x^2+1),\ q(x)=x^2+x+1
.$
\\
{\rm (iii)}
$k=6${\rm :}
$t(x)=1\pm 2x,\ r(x)=4x^2\pm 2x+1,\ q(x)=4x^2+1.$
\\
{\rm (iv)}
$k=10${\rm :}
$t(x)=10x^2+5x+3,\ r(x)=25x^4+25x^3+15x^2+5x+1,\ q(x)=25x^4+25x^3+25x^2+10x+3.$
\end{prop}
\end{rem}



\section{Preparation}
Let $k$ be a positive integer and $D$ a square-free positive integer.
In the rest of this paper, we suppose that $(t(x),r(x),q(x))$ parameterizes a complete family of elliptic curves with embedding degree $k$ and CM-discriminant $D$.
Especially, $r(x)$, $q(x)$ are non-constant irreducible.
Then the conditions in Definition \ref{family of p-f cv} (iii)(iv)(v) become
\begin{eqnarray}
&&
r(x) \mid \Phi_k(t(x)-1)
\label{1}
\\
&&
q(x)+1-t(x)=h(x)r(x)
\ \ {\rm for\ some\ non{\mbox{-}}zero\ }h(x) \in \Q[x]
\label{2}
\\
&&
Dy(x)^2=4h(x)r(x)-(t(x)-2)^2
\ \ {\rm for\ some\ }y(x) \in \Q[x].
\label{3}
\end{eqnarray}
We suppose that 
\begin{eqnarray*}\label{rho=1}
\rho:=\rho(t(x),r(x),q(x))=1.
\end{eqnarray*}
Then $\deg q=\deg r$.
Moreover, we obtain $\deg t <\deg r$ and $\deg y <\deg r$ by (\ref{rho,y,t}).
If $\deg h+\deg r <\deg t$, then 
$\deg q=\deg t$ by (\ref{2}).
This implies that 
$\deg h+\deg r <\deg t=\deg r$,
which is a contradiction.
On the other hand, if
$\deg h+\deg r =\deg t (\neq 0)$, then 
the leading coefficient of the right-hand side of (\ref{3}) is negative.
This contradicts that the leading coefficient of the left-hand side is positive.
Therefore we have
$\deg h+\deg r>\deg t$,
so that 
$\deg h=0$
by (\ref{2}).
Hence, we put $h:=h(x) \in \Q$ if $\rho=1$.

\begin{lem}[see Proposition 2.9 in \cite{FST}]
If $k=1$, then $\rho \ge 2$.
\end{lem}

\begin{proof}
By (\ref{1}), we see $t(x) \neq 0$ and so that
$1 \le \deg r \le \deg t$.
Hence by (\ref{2}), we obtain $\rho \ge 2$.
\end{proof}

\begin{lem}\label{sqrt{-D}}
Assume that $k \ge 2$.
Let $K$ be an algebraic number field which is isomorphic to $\Q[x]/(r(x))$.
If {\rm (\ref{1}), (\ref{2}), (\ref{3})} hold, then $K$ contains the imaginary quadratic field $\Q(\sqrt{-D})$.
Hence there exists a polynomial $e(x) \in \Q[x]$ such that 
$-D \equiv e(x)^2$ {\rm mod} $r(x)$. 
\end{lem}

\begin{proof}
If $y(x) \equiv 0$ mod $r(x)$, then $t(x)-1 \equiv 1$ mod $r(x)$ by (\ref{3}).
This implies that $k=1$, since $t(x)-1$ corresponds to a primitive $k$th root of unity in $\Q[x]/(r(x))$ by (\ref{1}).
Therefore $y(x) \not\equiv 0$ mod $r(x)$.
Then, $y(x)$ has an inverse $z(x)$ in $\Q[x]/(r(x))$ and
\begin{eqnarray*}
-D \equiv (t(x)-2)^2 z(x)^2 \mod r(x)
\end{eqnarray*}
 by (\ref{3}).
This implies that $\Q(\sqrt{-D}) \subset K$.
The second claim is obtained by $e(x):=(t(x)-2)z(x)$.
\end{proof}

It is well known (for example, see Boneh-Rubin-Silverberg \cite[Corollary 7.3]{BRS}) that
\begin{eqnarray}\label{equiv sqrt{-D} in cyclo fields}
\sqrt{-D} \in \Q(\mu_l) 
&\iff&
\begin{cases}
4 \mid l, D \mid \frac{l}{4}
\ \ 
{\rm or\ }
\\
4 \nmid l, D \mid l, D \equiv 3\ {\rm mod}\ 4.
\end{cases}
\end{eqnarray}

\begin{lem}\label{p-power}
Let $f(x) \in \Q[x]$ be a polynomial with degree $m$.
If the terms in $f(x)^2$ with degree not less than $m$ are in $\Q[x^a]$, then 
$f(x) \in \Q[x^a]$.
\end{lem}

\begin{proof}
The case when $a=1$ is trivial.
Hence we assume that $a \ge 2$, $m \ge 2$.
Put
$f(x)=f_mx^m+\cdots+f_1x+f_0$ $(f_i \in \Q)$.
First, we show
\begin{eqnarray}\label{subclaim}
f_{m-1}=\cdots = f_{m-(a-1)}=0.
\end{eqnarray}
From the term $x^{2m-1}$ in $f(x)^2$, we obtain
$2f_{m}f_{m-1}=0$.
Therefore 
$f_{m-1}=0$
since
$f_m \neq 0$.
If $a=2$, we obtain (\ref{subclaim}).
On the other hand, assume that $a>2$.
If we show 
$f_{m-1}=\cdots=f_{m-i}=0$
for some $1 \le i \le a-2$,
then we obtain $f_{m-(i+1)}=0$.
Indeed, we have
\[
\sum_{j=0}^{i+1}f_{m-j}f_{m-(i+1)+j}
=
f_mf_{m-(i+1)}+f_{m-1}f_{m-i}+\cdots+f_{m-(i+1)}f_{m}
=
0,
\]
which follows from the term $x^{2m-i-1}$ in $f(x)^2$.
This induces $f_{m-(i+1)}=0$.
Therefore, by induction with respect to $i$, we have (\ref{subclaim}) for $a>2$.
Assume $m=a$, then the proof is completed.

Next, suppose that $m>a$ and, for some $m'$ satisfying $0 \le m' \le \frac{m}{a}-1$, the equations
$f_{m-m{''}a-1}=\cdots = f_{m-m{''}a-(a-1)}=0$
hold if $0 \le m{''} \le m'$.
By the coefficients of $x^{2m-(m'+1)a-i}$ ($1 \le i \le a-1$), we have
\[
\sum_{j=0}^{m'a+a+i}f_{m-j}f_{m-(m'a+a+i)+j}=0.
\]
Note that $a \nmid i$.
Then, by induction with respect to $i$, we have
$f_{m-(m'+1)a-1}=\cdots=f_{m-(m'+1)a-(a-1)}$.
Consequently, we complete the proof by induction with respect to $m'$.
\end{proof}

\begin{lem}\label{k divided by p-power}
For $a,s,k \in \N$, 
$\Phi_{a^{s}k}(x)=\Phi_{ak}(x^{a^{s-1}})$.
\end{lem}

\begin{proof}
Denote a primitive $a^{s}k$th root of unity by $\zeta$.
Then $\zeta$ is a root of $\Phi_{ak}(x^{a^{s-1}})$.
So that the minimal polynomial $\Phi_{a^sk}(x)$ of $\zeta$ divides $\Phi_{ak}(x^{a^{s-1}})$.
Since both polynomials are monic and
$
\ds 
\varphi(a^{s}k)
=
a^{s}k\prod_{q \mid ak,\, q{\rm : prime}}({\textstyle 1-\frac{1}{q}})
=
a^{s-1}\varphi(ak)$, 
we obtain the claim.
\end{proof}

%
%



\section{Proof of Theorem \ref{main thm 1}}\label{k=prime section}
Suppose that $r(x)=\Phi_k(x)$.
%
Then we can see that $\rho \neq 1$ if $k=1,2$ by $\deg r=\varphi(k)$ and (\ref{rho,y,t}).
Hence we may assume that $\varphi(k) \neq 1$, i.e., $k \ge 3$.
Put $m:=\varphi(k)/2$, then $\deg r=2m$.

\begin{prop}\label{result deg t=1}
If $m=1$, i.e., $k=3,4,6$, then 
$\rho \neq 1$.
\end{prop}

\begin{proof}
Assume that $\rho = 1$.
Then we have $\deg t=1$ follows from (\ref{rho,y,t}) and $\deg t\ge 1$.
Put $X:=t(x)-1$.
Then, since $x$ and $X$ correspond to primitive $k$th root of unity, we have 
\begin{eqnarray*}
X
\equiv
\begin{cases}
x\ {\rm or}\ -x-1\ \ ({\rm if}\ k=3)
\\
x\ {\rm or}\ x-1\ \ ({\rm if}\ k=6)
\\
x\ {\rm or}\ -x\ \ ({\rm if}\ k=4)
\end{cases}
\ 
\mod r(x).
\end{eqnarray*}
In fact, the congruences are equal since $\deg t=1$.
Therefore 
$\Phi_k(x)=\Phi_k(X)$.
Since $X$ generates a power basis of $\Q(\mu_k)$ over $\Q$, we write $y(x)$ as a polynomial $y(X)$ of $X$ by abuse of notation.
Then we can replace (\ref{3}) as
\begin{eqnarray*}\label{3'}
Dy(X)^2=4h\Phi_k(X)-(X-1)^2.
\end{eqnarray*}

Now, we treat the case where $k=3$ (resp. $k=6$).
Since $\Phi_3(x)=x^2+x+1$ 
(resp. $\Phi_6(x)=x^2-x+1$),
we have
$
D(y_1X+y_0)^2=4h(X^2 \pm X+1)-(X^2-2X+1),
$
and hence
\[
\begin{cases}
Dy_1^2=4h-1,
\\
2Dy_1y_0=\pm 4h+2,
\\
Dy_0^2=4h-1.
\end{cases}
\]
These equations induce
$12h(h-1)=0$
(resp. $4h(3h-1)=0$).
So that we have $h=1$ (resp. $h=1/3$).
Thus we obtain 
$
q(x)=(X^2+X+1)+X=(X+1)^2
$
(resp. 
$
q(x)=\frac{1}{3}(X^2-X+1)+X=\frac{1}{3}(X+1)^2
$)
by (\ref{2}).
This contradicts the irreducibility of $q(x)$.
Next, we suppose that $k=4$.
Since $\Phi_4(x)=x^2+1$, in the same way as in the above, 
$
D(y_1X+y_0)^2=4h(X^2+1)-(X^2-2X+1).
$
Hence 
\[
\begin{cases}
Dy_1^2=4h-1,
\\
2Dy_1y_0=2,
\\
Dy_0^2=4h-1.
\end{cases}
\]
We have
$8h(2h-1)=0$ and $h=1/2$.
Thus
$
q(x)=\frac{1}{2}(X^2+1)+X=\frac{1}{2}(X+1)^2
$
by (\ref{2}).
This contradicts the irreducibility of $q(x)$ again.
\end{proof}

Next, we consider the case where $\varphi(k)/2 \ge 2$.

\begin{prop}\label{m ge 2,k:prime}
Suppose that $m \ge 2$.
If $k$ is described as $k=p$ or $k=2p$ by some odd prime number $p$, then
$\rho \neq 1$.
\end{prop}

\begin{proof}
It is sufficient to show that $\deg y>m$ if $k=p$ or $k=2p$ by $\deg r=2m$ and (\ref{rho,y,t}).
Since $x$ corresponds to a $k$th primitive root of unity, we have
$t(x)=x^g+1$ 
and
$\gcd(g,k)=1$.
Note that we know 
$p=D \equiv 3$ mod $4$
by 
Lemma \ref{sqrt{-D}} and (\ref{equiv sqrt{-D} in cyclo fields}).
In addition, note that $p \ge 7$ since $m \ge 2$.
Therefore, by \cite[Theorem 7 on p.349, Problem 8 on p.354]{BS}, we obtain
\[
\sqrt{-p}
=
\sum_{a=0}^{p-1}\chi_p(a)\zeta_p^a,
\]
where $\chi_p(a)$ stands for a Dirichlet character modulo $p$ of order $2$
and $\zeta_p \in \mu_p$.
In our case where $p$ is a prime number and satisfies that $p \equiv 3$ mod $4$, 
$\chi_p(a)$ is expressed in terms of the quadratic residue symbol $(\frac{*}{p})$ as 
\[
\chi_p(a)=
\begin{cases}
\left(\frac{a}{p}\right)
&
(p \nmid a)
\\
0
&
(p \mid a).
\end{cases}
\]
On the other hand, we claim that 
there exists some $\frac{p-1}{2}+1 \le b \le p-2$ such that 
$\chi_p(b-g)-\chi_p(b) \neq 0$.
Indeed, if there exists some $b$ such that $b-g \equiv 0$ mod $p$, 
then the claim is trivial.
Otherwise, if we assume that the claim does not hold, 
then the $\frac{p+1}{2}$ characters take a same value.
This contradicts the orthogonality of characters, since $p \ge 7$.


Suppose that $k=p$.
Then $y(x)$ corresponds to $\frac{(\zeta_p^g-1)\sqrt{-p}}{-p}$.
Moreover, 
\begin{eqnarray*}
(\zeta_p^g-1)\sqrt{-p}
&=&
\sum_{a=0}^{p-1}\chi_p(a)\zeta_p^{a+g}-\sum_{a=0}^{p-1}\chi_p(a)\zeta_p^{a}
\\
&=&
\sum_{a=0}^{p-1}(\chi_p(a-g)-\chi_p(a))\zeta_p^{a}
\\
&=&
\sum_{a=0}^{p-2}(\chi_p(a-g)-\chi_p(a))\zeta_p^{a}
+(1-\chi_p(g+1))\zeta_p^{p-1}.
\end{eqnarray*}
If $\chi_p(g+1)=1$, then the term $\zeta_p^b$ does not vanish.
This implies that the degree of $y(x)$ is over $m$.
On the other hand, assume that
$\chi_p(g+1)=0$, i.e., $g=p-1$.
Then, since
$
(1-\chi_p(g+1))\zeta_p^{p-1}
=
-(1+\zeta_p+\cdots+\zeta_p^{p-2}),
$
we obtain
\begin{eqnarray*}
(\zeta_p^g-1)\sqrt{-p}
=
\sum_{a=0}^{p-2}(\chi_p(a-g)-\chi_p(a)-1)\zeta_p^{a}.
\end{eqnarray*}
Hence, the term $\zeta_p^{p-2}$ does not vanish since
$\chi_p(p-2-g)-\chi_p(p-2)-1=\chi_p(2)-2 \neq 0$.
Combining $p \ge 7$ with this, we obtain $\deg y>m$ again.
(Alternatively, we obtain $\rho \neq 1$ by computing $\deg t$.)
For the last, assume that 
$\chi_p(g+1)=-1$.
Then 
\begin{eqnarray*}
(\zeta_p^g-1)\sqrt{-p}
=
\sum_{a=0}^{p-2}(\chi_p(a-g)-\chi_p(a)-2)\zeta_p^{a}.
\end{eqnarray*}
Therefore, we have only to show that there is an integer $2 \le i \le \frac{p-1}{2}$ such that
$\chi_p(p-i-g)-\chi_p(p-i) \neq 2$.
Assume that this does not hold.
Then we obtain $g=\frac{p-1}{2}$ and
\[
\begin{cases}
\chi_p(2)=\cdots=\chi_p(\frac{p-1}{2})=1
\\
\chi_p(\frac{p+1}{2})=\cdots=\chi_p(p-1)=-1
\end{cases}.
\]
This induces a contradiction.
Indeed, if $p=7,11$, then 
$\ds \left( \frac{4}{7} \right)=1$
and
$\ds \left( \frac{9}{11} \right)=1$.
Also, if $p \ge 19$, there exists an integer such that 
$\sqrt{\frac{p}{2}} <c<\sqrt{p}$.
Hence we obtain
$\ds \left( \frac{c^2}{p} \right)=1$ ($\frac{p+1}{2} \le c^2 \le p-1$).
Therefore $\deg y>m$ also holds.


Suppose that $k=2p$.
Then, using $\zeta_k=-\zeta_p$, we see that
\begin{eqnarray*}
(\zeta_k^g-1)\sqrt{-p}
&=&
-\sum_{a=0}^{p-1}\chi_p(a)\zeta_p^{a+g}+\sum_{a=0}^{p-1}\chi_p(a)\zeta_p^{a}
\\
&=&
\sum_{a=0}^{p-1}(\chi_p(a)-\chi_p(a-g))\zeta_p^{a}
\\
&=&
\sum_{a=0}^{p-2}(\chi_p(a)-\chi_p(a-g))\zeta_p^{a}
+(\chi_p(g+1)-1)\zeta_k^{p-1}.
\end{eqnarray*}
If $\chi_p(g+1)=1$, then the term $\zeta_k^b$ does not vanish, 
and the degree of $y(x)$ is over $m$.
On the other hand, assume that
$\chi_p(g+1)=0$, i.e., $g=p-1$ or $g=2p-1$.
Then, since
$\ds
\zeta_k^{p-1}
=
-\sum_{a=0}^{p-2}(-1)^a\zeta_k^a
$,
we obtain $\deg y>m$ in the same way as in the case when $k=p$.
Finally, we assume that $\chi_p(g+1)=-1$.
Then 
\begin{eqnarray*}
(\zeta_k^g-1)\sqrt{-p}
=
\sum_{a=0}^{p-2}(\chi_p(a)-\chi_p(a-g)+2(-1)^a)\zeta_k^{a}.
\end{eqnarray*}
If $p > 7$ (resp. $p=7$), then the term $\zeta_k^{p-4}$ (resp. $\zeta_k^{5}$) does not vanish.
Hence $\deg y>m$.
\end{proof}

In the same way, we obtain a result for certain composite numbers.

\begin{prop}\label{m ge 2,k=pq}
Let $p \ge 7$ be an odd prime number such that 
$p \equiv 3$ {\rm mod} $4$,
and
$Q \ge 2$ an integer.
Suppose that $k \in \{pQ,2pQ \}$, $D=p$, $t(x)=x+1$ and
$
(p-2)Q+1<\varphi(k).
$
Then $\rho \neq 1$.
\end{prop}

\begin{proof}
We show $\deg y>m$ if the conditions in the claim hold.
Let $k=pQ$.
Then 
\begin{eqnarray*}
(\zeta_k-1)\sqrt{-p}
&=&
\sum_{a=1}^{p-2}\chi_p(a)\zeta_p^{a}(\zeta_k-1)
+\chi_p(p-1)\zeta_p^{p-1}(\zeta_k-1)
\ \ \ ({\rm where}\ \zeta_p:=\zeta_{k}^Q)
\\
&=&
\sum_{a=1}^{p-2}\chi_p(a)(\zeta_k^{aQ+1}-\zeta_k^{aQ})
+(1+\zeta_p+\cdots+\zeta_p^{p-2})(\zeta_k-1)
\\
&=&
\sum_{a=1}^{p-2}(\chi_p(a)+1)\zeta_k^{aQ+1}
-
\sum_{a=1}^{p-2}(\chi_p(a)+1)\zeta_k^{aQ}.
\end{eqnarray*}
Moreover, there is $\frac{p-1}{2}+1 \le b \le p-2$ such that
$\chi_p(b-1)-\chi_p(b) \neq 0$.
Especially, not all of $\chi_p(a)$ ($\frac{p-1}{2} \le a \le p-2$) are equal to $-1$.
Since $aQ+1,aQ$ ($1 \le a \le p-2$) are different from each other and $aQ+1$ is less than $\varphi(k)$ by the assumption,
this means that $\deg y>m$.


If $k=2pQ$, since $\zeta_{2p}=-\zeta_p=\zeta_k^Q$, then 
\begin{eqnarray*}
(\zeta_k-1)\sqrt{-p}
&=&
-\sum_{a=1}^{p-2}\chi_p(a)\zeta_{2p}^a(\zeta_k-1)
+\chi_p(p-1)\zeta_{2p}^{p-1}(\zeta_k-1)
\\
&=&
-\sum_{a=1}^{p-2}(\chi_p(a)+(-1)^a)\zeta_{pQ}^{aQ+1}
+\sum_{a=1}^{p-2}(\chi_p(a)+(-1)^a)\zeta_{pQ}^{aQ}.
\end{eqnarray*}
Since $\chi(p-4)+(-1)^{p-4}=-2 \neq 0$ and $p \ge 7$, we have 
$\deg y \ge (p-4)Q+1>\frac{p-1}{2}Q$.
If $\gcd(p,Q)=1$, then $\frac{p-1}{2}Q \ge m$.
Otherwise, put $Q=p^s2^tQ'$ $\gcd(2p,Q')=1$, then
$
m=
\frac{\varphi(2pQ)}{2}
=
\frac{p-1}{2}p^s2^t\varphi(Q')
\le
\frac{p-1}{2}p^s2^tQ'
=
\frac{p-1}{2}Q.
$
Hence, in the same way, we complete the proof.
\end{proof}

From the proofs above, we can see that $\rho$-values are possible to be computed.
But they tend to be nearly $2$.




\section{Proof of Theorem \ref{main thm 2}}\label{p:cyclo poly,l-non-sporadic}
In this section, for an embedding degree $k$, we consider the case where $r(x)$ is taken as 
\[
r(x):=\Phi_{dk}(x).
\]
Denote by $\zeta_{dk} \in \mu_{dk}$ a primitive $dk$th root of unity corresponding to $x$.
Then $\zeta_k:=\zeta_{dk}^{d} \in \mu_k$  (resp. $\zeta_k^g$) corresponds to $x^{d}$ (resp. $t(x)-1=x^{dg}$).
Note that we assume $dg<\varphi(dk)$.
Then we consider
\[
\begin{cases}
r(x)=\Phi_{dk}(x) \mid \Phi_k(x^{dg}),
\ \ 
\gcd(g,k)=1,\ \ dg<\varphi(dk),
\\
Dy(x)^2=4h\Phi_{dk}(x)-(x^{dg}-1)^2,
\\
\rho = 1.
\end{cases}
\]
Then we have $2dg \le \varphi(dk)$.

First, we show the claim in Theorem \ref{main thm 2} (i), which says that
our consideration reduces to the case where $\gcd(d,k)=1$ and $d$ is square-free.
Assume that $\gcd(d,k)=:e \ge 2$ and write $d=ed'$.
Then 
\begin{eqnarray*}
Dy(x)^2
&=&
4h \Phi_{dk}(x)-(x^{dg}-1)^2
\\
&=&
4h \Phi_{d'k}(x^e)-((x^e)^{d'k}-1)^2.
\end{eqnarray*}
Moreover, we see that 
$d'g<\varphi(d'k)$
and
the condition (\ref{equiv sqrt{-D} in cyclo fields}) of $D$ does not change.
Applying Lemma \ref{p-power}, we obtain
$y(x) \in \Q[x^e]$.
Hence substitute $x$ for $x^e$, then
\begin{eqnarray*}
Dy(x)^2
=4h \Phi_{dk}(x)-(x^{d}-1)^2,
\ \ 
(d,k)=1.
\end{eqnarray*}
Moreover, if $d$ has a square factor.
Denote $d=a^2d'$ ($a \ge 2$), then  
\begin{eqnarray*}
Dy(x)^2
=
4h \Phi_{ad'k}(x^{a})-((x^{a})^{ad'}-1)^2.
\end{eqnarray*}
Applying Lemma \ref{p-power} again, we obtain
$y(x) \in \Q[x^{a}]$.
Again we see that 
$ad'g<\varphi(ad'k)$
and
the condition of $D$ does not change.
This implies that we may assume that $d$ is a square-free integer.
Therefore we have the claim above.
In the following, we assume that $d$ satisfies the condition.

\begin{prop}\label{m=2,p^a|k}
If $k=3,4,6$, then 
$\rho \neq 1$.
\end{prop}

\begin{proof}
We may remove the case $d=1$ since it is already shown in Lemma \ref{result deg t=1}.
If $d \ge 2$ and $\rho =1$, then 
$2 \varphi(k)=\varphi(dk) \ge 2dg \ge 2d$.
Hence $d=2$ and $g=1$, so that $k=3$ since $\gcd(d,k)=1$.
This contradicts $dg<\varphi(dk)$.
\end{proof}

Next, we show the claim in Theorem \ref{main thm 2} (ii):

\begin{prop}
If $k$ has a square-factor, then 
$\rho \neq 1$.
\end{prop}

Assume that $\rho =1$.
Denote
$k=a^2k'$, $a \ge 2$.
Put 
$m:=\frac{\varphi(dk)}{2}=a\frac{\varphi(d)\varphi(ak')}{2}$,
then $2dg \le 2m$ since $2m=\deg r$.
Then
\begin{eqnarray}\label{3'''}
Dy(x)^2
=
4h \Phi_{dak'}(x^a)-(x^{dg}-1)^2.
\end{eqnarray}
We have the following:

\begin{lem}\label{claim}
If $a \neq 2$, then
$m< 2dg <2m$.
\end{lem}

\begin{proof}
First, we show that
$\frac{\varphi(d)\varphi(ak')}{2} \in \Z$.
If $d=1$ and $\frac{\varphi(d)\varphi(ak')}{2} \notin \Z$,
then $a=2$, $k'=1$.
Hence $r(x)=\Phi_4(x)$ and $k=4$.
This contradicts our assumption $\rho=1$.
Hence we may suppose that $d \ge 2$.
In this case, we obtain $a \neq 2$ or $d \neq 2$ since $\gcd(d,a)=1$.
Thus $\frac{\varphi(d)\varphi(ak')}{2} \in \Z$.

Now assume that $dg=m$.
Then the above fact induces $a \mid dg$ 
which contradicts $\gcd(dg,a)$.
Therefore we obtain $dg<m$.

Second, to induce a contradiction, we assume that $a=2$.
Then
\[
Dy(x)^2
=
4h\Phi_{2dk'}(x^2)-(x^2)^{dg}+2x^{dg}-1
\]
and also $dg<m$, $a \mid 2dg$.
Hence, applying Lemma \ref{p-power} with $a=2$, we obtain
$y(x) \in \Q[x^2]$.
However, the term $x^{dg}$ in right-hand side of the above equation is not $0$.
Since $\gcd(dg,a)=1$ implies that $dg$ is odd, this is a contradiction.
So that we have $a \neq 2$.

Third, if $2dg < m$, then we obtain a contradiction from applying Lemma \ref{p-power} to (\ref{3'''}) and the fact $a \nmid dg$, in the same way.
Hence $m \le 2dg$.
Moreover, we see that
\[
2dg \neq m=a\varphi(d)\frac{\varphi(ak')}{2}.
\]
In fact, if it does not hold, $a \mid 2dg$
(note that $\frac{\varphi(ak')}{2} \in \Z$, since we already proved $a \neq 2$).
Since $\gcd(dg,a)=1$, this induces $a=2$ which is a contradiction.
Therefore we obtain $dg<m<2dg$.
This completes the proof.
\end{proof}

As in the same way as in the proof of Lemma \ref{p-power},
we compare the coefficients of the both sides of (\ref{3'''}), inductively.
By Lemma \ref{claim}, we have $2m>2dg$ and so that $\deg y=m$.
Denote
$y(x)=y_mx^m+\cdots+y_0$, $y_i \in \Q$, $y_m \neq 0$.
Again as in the same way as in the proof of Lemma \ref{p-power}, compare the both sides 
from the leading term to the term $x^{2d+1}$, inductively.
Then we have
\begin{eqnarray}\label{2d-m<i<m}
y_i=0,\ \ 2dg-m+1 \le i < m,\ a \nmid i
\end{eqnarray}
(note that $m<2dg$ by Lemma \ref{claim}).

Now suppose that $y_0 \neq 0$. 
Comparing the terms $X$ in (\ref{3'''}), we obtain 
$Dy_0y_1=0$,
and so that $y_1=0$.
Compare the both sides again from the constant term up to the term $X^{dg-1}$, inductively.
Then 
$y_i=0$, $0 < i \le dg-1,\ a \nmid i$.
Combining (\ref{2d-m<i<m}) with this, we have $y(x) \in \Q[x^a]$ since 
$dg \le m-1$.
Therefore a contradiction is induced from the terms $x^{dg}$ in both sides of (\ref{3'''}).

On the other hand, suppose that $y_0 = 0$.
Then, by (\ref{3'''}), we obtain $4h=1$.
Combining it with $\deg r=2m>2dg$, we have $Dy_m^2=1$ and $D=1$.
This implies that $\sqrt{-1} \in \Q(\mu_{dg})$, and so that $4 \mid dk$.
If $2 \mid d$, then $2 \mid k$ since $d$ is square-free.
This contradicts with $\gcd(d,k)=1$.
Hence $4 \mid a^2k'$ and, moreover, we see that $a$ is even since $k'$ is square-free.
Therefore, $dg$ is odd.
Finally, by comparing the coefficients of $x^2,x^4,\ldots,x^{2dg-2}$ of (\ref{3'''}) in ascending order of powers, we obtain 
\begin{eqnarray*}
y_i=0\ \ (0 < i \le dg-1,\ a \nmid i).
\end{eqnarray*}
Hence, as in the same way as above, a contradiction is induced.

\vspace{0.3cm}

It seems to that the most interesting case for cryptography is the case where the embedding degree $k$ consists of powers of $2$ or $3$.
So that, it is one of the important problems to give the explicit formulas of $\rho$-values in the case where Theorem \ref{main thm 2} (ii) holds.
Also, to give bounds of $\rho$-values in many other cases is needed for applications.


\vspace{0.3cm}
{\bfseries Acknowledgement}. The author would like to express his gratitude for
a number of helpful suggestions to Naoki Kanayama.



\vspace{4mm}
\noindent 
Keiji Okano\\
Department of Mathematics, \\
Faculty of Science and Technology, \\
Tokyo University of Science \\
2641 Yamazaki, Noda, Chiba, 278-8510 Japan
%
%
\\
{\it E-mail address}: {\tt okano\_keiji@ma.noda.tus.ac.jp}, 

\end{document}